\documentclass[12pt]{article}
\usepackage{amsmath,amsthm,amsfonts,latexsym,amssymb,mathrsfs,graphics,pst-all}

 \usepackage{footnote}
\theoremstyle{plain}
\newtheorem{theorem}{Theorem}[section]
\newtheorem{lemma}[theorem]{Lemma}
\newtheorem{corollary}[theorem]{Corollary}
\newtheorem{proposition}[theorem]{Proposition}

\theoremstyle{definition}
\newtheorem{definition}[theorem]{Definition}
\newtheorem{example}[theorem]{Example}

\numberwithin{equation}{section}

 \newcommand{\bnum}{\begin{enumerate}}
 \newcommand{\enum}{\end{enumerate}}

\begin{document}

\begin{center}{\bf \large{An Extension of Semicommutative Rings via Reflexivity }}\\
\end{center}

\begin{center}
 {\textbf{Sanjiv Subba}\\ Department of Mathematics, National Institute of Technology Meghalaya, Shillong-793003, India.\\  Email: sanjivsubba59@gmail.com\\ 
\textbf{Tikaram Subedi}\\ Department of Mathematics, National Institute of Technology Meghalaya, Shillong-793003, India. \\  Email: tikaram.subedi@nitm.ac.in}
\end{center}

\begin{small}
\textbf{Abstract}:\textit{ 
	This article introduces the notion of an NJ-reflexive ring and demonstrates that it is distinct from the concept of a reflexive ring. The class of NJ-reflexive rings contains the class of semicommutative rings, the class of left (right) quasi-duo rings, and the class of  J-clean rings but is strictly larger than these classes. Additionally, the article investigates a sufficient condition for NJ-reflexive rings to be left (right) quasi-duo, as well as some conditions for NJ-reflexive rings to be reduced. It also explores extensions of NJ-reflexive rings and notes that the NJ-reflexive property may not carry over to polynomial extensions.}


\end{small}

\section{Introduction}
\label{s1}
The concept of normal subgroups in group theory is well-known. A subgroup S of a group G is normal if and only if for any $x, y \in G$, $xy \in S$ implies that $yx \in S$.  In \cite{Th}, this property, as extended to arbitrary subsets of semi-groups and rings, was called r{\'e}flectif. This concept was further extended to an ideal of a ring by Mason in \cite{Mason}, where he introduced \textit{reflexive ideals} for the first time. A right ideal $I$ of ring $R$ is reflexive if $yRx\subseteq I$ whenever $xRy \subseteq I$, for $x, y \in R$. A ring $R$ is said to be reflexive if its zero ideal is reflexive.

\quad Recall that a ring $R$ is called   \textit{semicommutative} (\cite{sem})  if $xy=0$ implies $xRy=0$ for any $x, y\in R$ or equivalently, every left (right) annihilator of any element of $R$ is an ideal.  $R$ is  \textit{nil-semicommutative} if $xy\in N(R)$ implies $xRy\subseteq N(R)$ for $x,y\in R$;  \textit{generalized weakly symmetric } (\cite{ gwsr})  if for any $x,y,z\in R$, $xyz=0$ implies $yxz\in N(R)$. In order to study  these sort of classes of rings, reflexive ring property and semicommutative ring property  have been  studied extensively by mathematicians for several decades (see \cite{rpr}, \cite{sem}, \cite{Mason}), and they continue to be an area of active research. Kwak and Lee in \cite[Proposition 2.2]{rpr} showed a connection between semicommutativity and reflexivity. However, ``semicommutativity'' and ``reflexivity'' of a ring are two independent notions (see \cite[Example 2.3]{rpr}).

\quad Motivated by this, we introduce the concept of NJ-reflexive rings by leveraging the idea of reflexive ring property to generalize semicommutative rings. Interestingly, this class of rings also turns out to be a generalization of right (left) quasi-duo rings and J-clean rings.

\quad In this paper, $R$ represents an associative ring with unity, and all modules  are unital. We adopt the following notations: $U(R)$ for the set of all units of $R$, $Z(R)$ for the set of all central elements of $R$, $E(R)$ for the set of all idempotent elements of $R$, $J(R)$ for the Jacobson radical of $R$, $N(R)$ for the set of all nilpotent elements of $R$, $M_n(R)$ ($T_n(R))$ for the ring of all $n \times n$ matrices (upper triangular matrices) over $R$, and $N^*(R)$ for the upper nil radical of $R$. Moreover, we use the notation $E_{ij}$ for the matrix unit in $M_n(R)$ whose $(i, j)^{th}$ entry is $1$ and zero elsewhere. For any $a \in R$, $l(a)$ and $r(a)$ respectively denote the left and right annihilator of $a$, and $\langle a\rangle$  stands for the ideal generated by $a$.

\section{NJ-reflexive Rings}
This section begins with the introduction of a new class of rings called NJ-reflexive rings. It further explores some fundamental properties of NJ-reflexive rings, followed by investigations of various subclasses.

\begin{definition}
	Let $x,y\in R$. We call $R$  $NJ-reflexive$ if  $xRy\subseteq N(R)$ implies $yRx\subseteq J(R)$.  
\end{definition}
The definition of NJ-reflexive rings may prompt an investigation into the correlation between the ``reflexivity'' and ``NJ-reflexivity'' of a ring. However, these notions  are independent (see Example \ref{rem}  and Example \ref{utm}).
\begin{example}\label{rem}
	For any domain $R$, $M_n(R)$ is reflexive (by \cite[Theorem 2.6 (2)]{rpr}). Now, observe that $E_{11}M_n(R)E_{nn}\subseteq N(M_n(R))$. Note that we also have $E_{nn}M_n(R)E_{11}\not\subseteq J(M_n(R)$ since $E_{nn}E_{n1}E_{11}=E_{n1}\notin J(M_n(R))$, that is, $M_n(R)$ is not NJ-reflexive.

\end{example}

A ring $R$ is termed  \textit{directly finite} if the condition $xy = 1$ implies $yx = 1$, for $x,y\in R$.

\begin{proposition}\label{tri}
	Let $R$ be an NJ-reflexive ring. Then:
	\begin{enumerate}
		
		\item \label{nii} If $xRy\subseteq  N(R)$, then either $x\in M$ or $y\in M$ for any maximal left ideal $M$ of $R$.                         
		
		\item $R$ is directly finite.
	\end{enumerate}
\end{proposition}                                                      
\begin{proof}
	\begin{enumerate}

		\item Let $M$ be a maximal left ideal of $R$ and $xRy\subseteq  N(R)$ for $x,y\in R$. Suppose $x\notin M$. This implies that $Rx + M = R$. So, $rx + m = 1$ for some $r\in R$ and $m\in M$. Hence, we obtain $yrx + ym = y$. As $R$ is NJ-reflexive, $yrx\in  J(R)$, implying that $y\in M$.

		\item \label{df} 
		Suppose $x,y \in R$ satisfy $xy = 1$. Observe that $(1-yx)Ryx\subseteq N(R)$. Since $R$ is NJ-reflexive, we have $yxy(1-yx) = y(1-yx)\in J(R)$. Therefore, $xy(1-yx)=1-yx\in J(R)\cap E(R)$. So, $yx = 1$.
	\end{enumerate}
\end{proof}

It is easy to observe that the proof of Lemma \ref{equi} is obvious.

\begin{lemma}\label{equi}
	The following are equivalent:
	\begin{enumerate}
		\item $xyz\in N(R)$ implies $yxz\in J(R)$ for any $x,y,z\in R$.
		\item $xyz\in N(R)$ implies $xzy\in J(R)$ for any $x,y,z\in R$.
		\item $xyz\in N(R)$ implies $zyx\in J(R)$ for any $x,y,z\in R$.
	\end{enumerate}
\end{lemma}

\begin{lemma}\label{rp}
	If $R$ is a ring in which $xyz\in N(R)$ implies $yxz\in J(R)$ for $x,y,z\in R$, then $R$ is NJ-reflexive.	
\end{lemma}
\begin{proof}
	Let $x,y\in R$ be such that $xRy\subseteq N(R)$. So, $yxr\in N(R)$ for any $r\in R$. By Lemma \ref{equi}, $yrx\in J(R)$. Hence, $R$ is NJ-reflexive.
\end{proof}

\begin{theorem}\label{bac}
	If $R$ is a ring satisfying $xyz\in N(R)$ implies $yxz\in N(R)$ for $x,y,z\in R$, then $R$ in NJ-reflexive.
\end{theorem}
\begin{proof}
	Let $x,y,z\in R$ be such that $(xyz)^n=0$, where $n$ is a positive integer. Then, $\underbrace{ (xyz)(xyz)...(xyz)}_\text{$n$-times}r^n=0$ for any $r\in R$. So, $xyzr\underbrace{(xyz)(xyz)....(xyz)}_\text{$n-1$-times}r^{n-1}\in N(R)$, which leads to  $yxzr\underbrace{(xyz)(xyz)...(xyz)}_\text{$n-1$-times}r^{n-1}\in N(R)$. This further implies that $xyzryxzr\underbrace{(xyz)(xyz)...(xyz)}_\text{$n-2$-times}r^{n-2}\in N(R)$. So, we obtain \\ $yxzryxzr\underbrace{(xyz)(xyz)...(xyz)}_\text{$n-2$-times}r^{n-2}\in N(R)$. Proceeding similarly, we obtain that $(yxzr)^n\in N(R)$. So, $yxz\in J(R)$. Therefore, by Lemma \ref{rp}, $R$ is NJ-reflexive.
\end{proof}

\begin{corollary}
	Nil-semicommutative rings are  NJ-reflexive. Thereby, semicommutative rings are NJ-reflexive.
\end{corollary}
\begin{proof}
	Let $x,y,z\in R$ be such that $xyz\in N(R)$. Since $R$ is nil-semicommutative, $xyxz\in N(R)$. Similarly, we obtain $(yxz)^2\in N(R)$. So, by Theorem \ref{bac}, $R$ is NJ-reflexive.  
\end{proof}

\begin{theorem}\label{qnjs}
	
	Generalized weakly symmetric rings in which the index of nilpotent elements is at most 2 are NJ-reflexive.

\end{theorem}
\begin{proof}

	Let $xyz\in N(R)$. Then, $yxrzy(xyz)^2yxrzy=0$ for any $r\in R$. Since $R$ is GWS, $xyzyxrzyxyzyxrzy\in N(R)$. So, $xyzyxrzy\in N(R)$, that is, $yzyxrzyx\in N(R)$. So, $rz(yzyxrzyx)^2rz=0$. This implies that $yzyxrzyxrzyzyxrzyxrz\in N(R)$, that is, $(zyxr)^2zy\in N(R)$. Since $xr((zyxr)^2zy)^2xr=0$, $zyxr\in N(R)$. Hence, $zyx\in J(R)$. By Lemma \ref{rp} and Lemma \ref{equi}, $R$ is NJ-reflexive. 
\end{proof}

A ring $R$ is said to be:
\begin{enumerate}
	\item \textit{left (right) quasi-duo} (\cite{1s2id}) if every maximal left (right) ideal of $R$ is an ideal of $R$.
	\item 	\textit{left} (\textit{right}) $SF$ (\cite{sfr})   if all simple left (right) $R$-modules are flat.
\end{enumerate}

It is well known that a regular ring (that is, a ring $R$ in which for each $x\in R$, $x=xrx$ for some $r\in R$) is left-SF. However, it still remains an unsolved question whether a left SF-ring is regular.

\begin{theorem}\label{qduo}
	If  $R$ is a left (right) quasi-duo ring, then $R$  is NJ-reflexive. The converse holds if $R$ is a left (right) SF-ring.   
\end{theorem}
\begin{proof}
	Suppose $xRy\subseteq N(R)$ for some $x,y\in R$. Assume that $yr_0x\notin J(R)$ for some $r_0\in R$. Then, $Ryr_0x+M=R$ for some maximal left ideal $M$ of $R$. Hence, $ryr_0x+m=1$ for some $r\in R, m\in M$. This implies that $ryr_o(xry)r_ox=(1-m)^2$. According to \cite[Lemma 2.3]{oqdr} and hypothesis, it follows that $xry\in J(R)\subseteq M$. Therefore, $1\in M$, which leads to a contradiction. Thus, $yRx\subseteq J(R)$. Similar proof applies if $R$ is a right quasi-duo ring.
	
	Assuming $R$ is an NJ-reflexive left SF-ring, it follows from \cite[Proposition 3.2]{rnsf} that $R/J(R)$ is left SF. Suppose there exists $a_0\in R$ such that $a_{0}^2\in J(R)$ and $a_0\notin J(R)$. If $Rr(a_0)+J(R)=R$, then $1=\sum\limits^{finite}r_is_i+j$ for $j\in J(R)$, $r_i\in R$, and $s_i\in r(a_0)$. Thus, $a_0=ja_0+\sum\limits^{finite}r_is_ia_0$. Note that $s_ia_0Rs_ia_0\subseteq N(R)$ and $R$ being NJ-reflexive implies that $s_ia_0Rs_ia_0\subseteq J(R)$. This yieds that $s_ia_0\in J(R)$. So, $a_0=ja_0+\sum\limits^{finite}r_is_ia_0\in J(R)$. But this is a contradiction. Therefore, $Rr(a_0)+J(R)\neq R$. Thus,  there is a maximal left ideal $L$ such that $Rr(a_0)+J(R)\subseteq L$. Since $a_0^2\in L$, by \cite[Lemma 3.14]{rnsf}, there exists $x\in L$ such that $a_0^2=a_0^2x$. It gives that $a_0-a_0x\in r(a_0)\subseteq L$, which further implies that $a_0\in L$. Thus, there exists $y\in L$ such that $a_0=a_0y$. That is, $1-y\in r(a_0)\subseteq L$. However, this implies that $1\in L$, which is a contradiction. Hence, $R/J(R)$ is reduced (that is,  only $0$ is the nilpotent element). Hence, $R/J(R)$ is strongly regular by \cite[Remark 3.13]{rnsf}, and so $R$ is left quasi-duo.
\end{proof}

$R$ is called \textit{strongly regular} if for each $x\in R$, $x=x^2r$ for some $r\in R$. Clearly, strongly regular rings are regular.

\begin{corollary}
	If $R$ is an NJ-reflexive left SF-ring, then $R$ is strongly regular.
\end{corollary}
\begin{proof}
	By Theorem \ref{qduo}, $R$ is left quasi-duo. Thus, by \cite[Theorem 4.10]{rnsf}, $R$ is strongly regular.
\end{proof}

There exists an NJ-reflexive ring which is not left (right) quasi-duo, as illustrated by the following  example.

\begin{example}\label{eqd}
	By \cite[Example 2(ii)]{1s2id}, $\mathbb{H}[x]$ is not right quasi-duo, where $\mathbb{H}$ is the Hamilton quaternion over the field of real numbers. However,
	since $\mathbb{H}[x]$ is a reduced ring, it is NJ-reflexive. 
\end{example}
$R$ is said to be \textit{J-clean} if for each $x\in R$, $x=e+j$ for some $e\in E(R)$ and $j\in J(R)$. By \cite[Corollary 4.5]{jclean}, $T_n(\mathbb{Z}_2)$ is J-clean for any positive integer $n$. 
\begin{theorem}\label{fj}
	J-clean rings are NJ-reflexive.
\end{theorem}
\begin{proof}
	Assume $x, y\in R$ satisfy $xRy\subseteq N(R)$. As $R$ is J-clean, for any $r\in R$, $yrx-e=j$ for some $e\in E(R)$ and $j\in J(R)$. Observe that $j^2=yr(xy)rx-yrxe-ej$. Since $xy\in N(R)$, $1-xy\in U(R)$. There exists $e_1\in E(R)$ such that $1-xy-e_1\in J(R)$. So, $(1-xy)^{-1}e_1\in U(R)$ which further implies that $e_1\in U(R)$, that is, $e_1=1$. Therefore, $xy\in J(R)$ which yields that $yrxe\in J(R)$ (as $ j^2=yr(xy)rx-yrxe-ej$). As $yrxe-e=je$, we have $e\in J(R)\cap E(R)$. It follows that $e=0$. Since $yrx-e=j$, $yrx=j\in J(R)$. Hence $yRx\subseteq J(R)$.
\end{proof}
Any field $F$ with at least three element is not J-clean. Therefore, the converse of  Theorem \ref{fj} is not true. 

\quad Let $ME_l(R)=\{e\in E(R):Re$ is a~minimal~left~ideal~of~$R\}$. An element $e\in E(R)$ is called \textit{left (right) semicentral} if $xe=exe (ex=exe)$ for any $x\in R$. Then, $R$ is said to be:
\begin{enumerate}
	\item \textit{left min-abel} (\cite{cnil}) if every element of $ME_l(R)$ is left  semicentral in $R$.
	\item left $MC2$ (\cite{cnil}) if $aRe=0$ implies $eRa=0$ for any $a\in R,e\in ME_l(R)$.
\end{enumerate}

\begin{lemma}\label{min}
	NJ-reflexive rings are left-min abel.
\end{lemma}
\begin{proof}

	Assume $e \in M E_l(R)$ and $a\in R$. Let $x = ae - eae$. Then $ex = 0$, $xe = x$ and $x^2 = 0$. Note that $xRx \subseteq N(R)$. As $R$ is NJ-reflexive, we have $xRx \subseteq J(R)$. Since $J(R)$ is semiprime, we have $x\in J(R)$. If $x=0$, then we are done. Otherwise, since $Re$ is a minimal left ideal of $R$, we have $Re = Rx$. As $x \in J(R)$, we have $Re = Rx \subseteq J(R)$, which implies $e=0$, a contradiction. Thus, $R$ is left min-abel.
\end{proof}
A left $R$-module $M$ is called \textit{Wnil-injective} (\cite{nir}) if for each $a(\neq 0)\in N(R)$, there is a positive integer $k$ satisfying $a^k\neq 0$ and each left $R$-homomorphism from $Ra^k$ to $M$ can be extended to one from $R$ to $M$. 

\quad While it is clear that reduced rings are NJ-reflexive, the reverse statement does not hold in general (for example, $T_2(\mathbb{Z}_2)$). Therefore, it is worthwhile to investigate some conditions under which an NJ-reflexive ring is necessarily reduced.

\begin{proposition}
	If $R$ is an NJ-reflexive ring, then each of the following conditions imply that $R$ is reduced:
	\begin{enumerate}
		\item R is semiprimitive.
		\item  R is a left MC2  ring, and each simple singular left $R$ module is Wnil-injective.
		
	\end{enumerate}
\end{proposition}
\begin{proof}
	\begin{enumerate}
		\item Suppose $x^2=0$ for some $x\in R$. Since $R$ is NJ-reflexive, $xRx\subseteq J(R)$.  $J(R)$ being semiprime, $x\in J(R)=0$.
		\item \label{mc2nj}  Suppose $x^2=0$ for some $x(\neq 0)\in R$. Then, $l(x)\subseteq M$ for some maximal left ideal $M$. Assume that $M$ is not an essential left ideal. Then, $M=l(e)$ for some $e\in ME_l(R)$.  By Lemma \ref{min}, $R$ is left min-abel, and since $R$ is left $MC2$, by \cite[Theorem 1.8]{cnil}, $e\in Z(R)$. Thus, $ex=0$, which implies $e\in l(x)\subseteq M=l(e)$. But, this is a contradiction. Thus, $M$ is essential, and $R/M$ is a simple singular left $R$ module. Moreover $R/M$ is Wnil-injective (by hypothesis). Let $\sigma:Rx\rightarrow R/M$ be a left $R$-homomorphism defined as $\sigma (rx)=r+M$. Since $R/M$ is Wnil-injective,  $1-xt\in M$ for some $t\in R$. Because $R$ is NJ-reflexive and $xRx\subseteq N(R)$, we have $xRx\subseteq J(R)$, and hence $x\in J(R)$. So, $1-xt\in U(R)\cap M$, a contradiction. Therefore, $x=0$.
		
	\end{enumerate}
\end{proof}

\begin{lemma}\label{rqj}
	If $I$ is an ideal of $R$ that is contained in $J(R)$, and $R/I$ is NJ-reflexive, then $R$ is also NJ-reflexive.
\end{lemma}
\begin{proof}
	Let $\overline{R}=R/I$. Suppose $xRy\subseteq N(R)$ for some $x,y\in R$. It can be observed that $\bar{x}\overline{R}\bar{y}\subseteq N(\overline{R})$. Since $\overline{R}$ is NJ-reflexive, $\Bar{y}\overline{R}\Bar{x}\subseteq J(\overline{R})=J(R)/I$. Let $r$ be any element of $R$. Then, for any $z\in R$, $\Bar{1}-\Bar{y}\Bar{r}\Bar{x}\Bar{z}\in U(\overline{R})$, which means that $(\Bar{1}-\Bar{y}\Bar{r}\Bar{x}\Bar{z})\Bar{u}=\Bar{1}=\Bar{u}(\Bar{1}-\Bar{y}\Bar{r}\Bar{x}\Bar{z})$ for some $u\in R$. Therefore, $1-(1-yrxz)u\in I$. As $I\subseteq J(R)$, it follows that $(1-yrxz)u\in U(R)$. Hence, $yrx\in J(R)$, which implies that $yRx\subseteq J(R)$.
\end{proof}

We refer to \cite[Example 3]{yh} to show that the converse of Lemma \ref{rqj} does not hold.

\begin{example}
	
	Let $R$ the localization of  $\mathbb{Z}$ at $3\mathbb{Z}$, and let $S$ be the quaternions over $R$.  It is clear that $S$ is a non-commutative domain, and thus an NJ-reflexive ring. By \cite[Example 3]{yh}, we have $J(S)=3S$ and $S/3S\cong M_2(\mathbb{Z}_3)$. However, $M_2(\mathbb{Z}_3)$ is not NJ-reflexive (by Example \ref{rem}). 
	
\end{example}

\begin{theorem}\label{nilfrac}
	If $I$ is a nil ideal of $R$, then $R$ is NJ-reflexive if and only if $R/I$ is NJ-reflexive.
\end{theorem}
\begin{proof}

	Assume that $R$ is NJ-reflexive and $I$ is a nil ideal of $R$. We  write $\overline{R}=R/I$. Let $\Bar{x}, \Bar{y}\in \overline{R}$ such that $\Bar{x}\overline{R}\Bar{y}\subseteq N(\overline{R})$. It follows that $xRy\subseteq N(R)$. Since $R$ is NJ-reflexive, $yRx\subseteq J(R)$. Therefore, for any $r\in R$, we have $\Bar{y}\Bar{r}\Bar{x}\in J(R)/I=J(\overline{R})$, which implies that $\Bar{y}\overline{R}\Bar{x}\subseteq J(\overline{R})$.
	
	Suppose that $R/I$ is NJ-reflexive, where $I$ is a nil ideal of $R$. By Lemma $\ref{rqj}$, it follows that $R$ is  NJ-reflexive. 
\end{proof}
\begin{corollary}
	Let $R$ be a ring with nil Jacobson radical. Then, $R$ is NJ-reflexive if and only if $R/J(R)$ is NJ-reflexive. 
	
\end{corollary}

\begin{proposition}\label{sdp}
	Every finite subdirect product of NJ-reflexive rings is NJ-reflexive.
\end{proposition}

\begin{proof}

	Let $L$ and $K$ be ideals of $R$ such that $L \cap K = 0$ and  $R/L$ and $R/K$  are NJ-reflexive. Suppose $xRy \subseteq N(R)$ for some $x,y \in R$. By hypothesis, $(y + L)(R/L)(x + L) \subseteq J(R/L)$ and $(y + K)(R/K)(x + K) \subseteq J(R/K)$. Let $r$ be an arbitrary element of $R$. For each $t \in R$, $(1 - yrxt + L) \in U(R/L)$ and $(1 - yrxt + K) \in U(R/K)$. As a consequence, we have that $1 - (1 - yrxt)w \in L$ and $1 - (1 - yrxt)z \in K$ for some $w, z \in R$. Observe that $(1 - (1 - yrxt)w)(1 - (1 - yrxt)z) \in LK\subseteq L \cap K = 0$. Hence $1 - (1 - yrxt)s = 0$ for some $s \in R$, that is, $(1 - yrxt)s = 1$. Therefore, $yRx \subseteq J(R)$.
\end{proof}

\begin{proposition}\label{intk}
	Let $K$ and $L$ be ideals of R such that $R/K$ and $R/L$ are NJ-reflexive. Then, $R/({L\cap K})$ is NJ-reflexive.
\end{proposition}
\begin{proof}
	
	Define $\sigma : R/(L \cap K) \rightarrow R/L$ and $\Phi : R/(L \cap K) \rightarrow R/K$ by $\sigma (r + L \cap K) = r + L$ and $\Phi(r + L \cap K) = r + K$, respectively. It is clear that both $\sigma$ and $\Phi$ are epimorphisms with $\text{ker}(\sigma) \cap \text{ker}(\Phi) = 0$. Therefore, $R/(L \cap K)$ is the subdirect product of $R/L$ and $R/K$. By Proposition \ref{sdp}, $R/(L \cap K)$ is NJ-reflexive.
\end{proof}  

\begin{proposition}\label{il}
	If $K$ and $L$ are ideals of R such that $R/K$ and $R/L$ are NJ-reflexive, then $R/{LK}$ is NJ-reflexive. 
\end{proposition}
\begin{proof}
	
	We can observe that $LK \subseteq L \cap K$ and $R/(L \cap K) \cong (R/LK)/((L \cap K)/LK)$. By Proposition \ref{intk}, $R/(L \cap K)$ is NJ-reflexive.  Since $((L \cap K)/LK)^2 = 0$, by Theorem \ref{nilfrac}, $R/LK$ is also NJ-reflexive.
\end{proof}

We can derive the following corollary as a consequence of Proposition \ref{il}.
\begin{corollary}
	The following are equivalent for an ideal I of R:
	\begin{enumerate}
		\item R/I is NJ-reflexive.
		\item $R/{I^n}$ is NJ-reflexive any  positive integer n.
	\end{enumerate}
\end{corollary}

\begin{proposition}\label{dp}
	Let $\{R_{\alpha} \}_{\alpha \in \Gamma}$ be a family of rings for an index set $\Gamma$. Then $\Pi_{\alpha\in \Gamma}R_{\alpha}$ is NJ-reflexive if and only if $R_{\alpha}$ is NJ-reflexive for each $\alpha \in \Gamma$.
\end{proposition}
\begin{proof}
	Suppose that $R_\alpha$ is NJ-reflexive for each $\alpha\in \Gamma$ and $(x_i)_{i\in \Gamma}(\Pi_{\alpha\in \Gamma}R_{\alpha})(y_i)_{i\in \Gamma}\subseteq N(\Pi_{\alpha\in \Gamma}R_{\alpha})$ for some $(x_i)_{i\in \Gamma}, (y_i)_{i\in \Gamma}\in \Pi_{\alpha\in \Gamma}R_{\alpha}$. This implies that $x_\alpha R_\alpha y_\alpha\subseteq N(R_\alpha).$ By hypothesis, $y_\alpha R_\alpha x_\alpha\subseteq J(R_\alpha)$, which implies that $(y_i)_{i\in \Gamma}(\Pi_{\alpha\in \Gamma}R_{\alpha})(x_i)_{i\in \Gamma}\subseteq J(\Pi_{\alpha\in \Gamma}R_{\alpha})=\Pi_{\alpha\in \Gamma}J(R_\alpha)$. It follows that  $\Pi_{\alpha\in \Gamma}R_{\alpha}$ is NJ-reflexive. The converse is obvious.\\

\end{proof}

\begin{corollary}
	eR and $(1-e)R$  are NJ-reflexive for some central idempotent $e\in R$ if and only if $R$ is NJ-reflexive.
\end{corollary}

\begin{proposition}\label{eRe}
	R is NJ-reflexive if and only if $eRe$ is NJ-reflexive for all $e\in E(R)$.
\end{proposition}
\begin{proof}
	
	Assuming that $R$ is NJ-reflexive and $(eae)(eRe)(ebe)\subseteq N({eRe})$ for some $eae, ebe\in eRe$. Then $(ebe)(eRe)(eae)\subseteq J(R)$, and due to the fact that $eJ(R)e=J(eRe)$, we have $(ebe)(eRe)(eae)\subseteq J({eRe})$. Therefore, $eRe$ is NJ-reflexive. The converse statement is self-evident and requires no further explanation.
\end{proof}

\begin{proposition}\label{upt}
	The following are equivalent:
	\begin{enumerate}
		\item $R$ is an NJ-reflexive ring.
		\item $T_n(R)$ is NJ-reflexive for any positive integer $n$. 
		
	\end{enumerate}    
\end{proposition}

\begin{proof}
	$\mathbf{(1)\implies (2)}.$
	Take $I=\begin{pmatrix}
		0 & a_{12}  & \dots & a_{1n} \\
		0 & 0 & \dots & a_{2n} \\
		\vdots & \vdots & \ddots & \vdots \\
		0 & 0 & \dots & 0 
	\end{pmatrix}$, which is a nil ideal of $T_n(R)$. Observe that $T_n(R)/I \cong R\times R\times \dots \times R$. By Theorem \ref{nilfrac} and Proposition \ref{dp}, $T_n(R)$ is NJ-reflexive.   \\
	Proof of $\mathbf{(2)\implies (1)}$, follows from Proposition \ref{eRe}.
	
\end{proof}

\begin{example}\label{utm}
	
	Let $D$ be a domain. By Proposition \ref{upt}, $T_2(D)$ is NJ-reflexive which is not reflexive (by \cite[Example 2.7]{rpr})

\end{example}

\quad Let $M$ be an $(R,R)$-bimodule. Then, the \textit{trivial extension} of $R$ by $M$ is defined by $T(R,M) = R \oplus M$, where the addition is the usual  and the multiplication is defined by $(r_1,s_1)(r_2,s_2) = (r_1r_2,r_1s_2+s_1r_2)$ for $s_i\in M$ and $r_i\in R$. The ring $T(R,M)$ is isomorphic to the ring $\left\{\left(\begin{array}{rr}
	t & s\\
	0 & t
\end{array}
\right) : t\in R, s\in M \right\}$, where the operations are usual matrix operations.

\begin{proposition}
	Let $M$ be an $(R,R)$-bimodule. Then the following are equivalent:
	\begin{enumerate}
		\item The trivial extension $T(R, M)$ is NJ-reflexive.
		\item $R$ is NJ-reflexive.
	\end{enumerate}
\end{proposition}

\begin{proof} Denote $\Gamma=T(R, M)$.\\
	Assume that $\Gamma$ is NJ-reflexive. Let $xRy\subseteq N(R)$ for some $x,y\in R$. Then  $\begin{pmatrix}
		x & 0\\
		0 & x
	\end{pmatrix}\Gamma \begin{pmatrix}
		y & 0\\
		0 & y
	\end{pmatrix}\subseteq N(\Gamma)$. Since $\Gamma$ is NJ-reflexive, we have $\begin{pmatrix}
		y & 0\\
		0 & y
	\end{pmatrix}\Gamma \begin{pmatrix}
		x & 0\\
		0 & x
	\end{pmatrix}\subseteq J(\Gamma)$. Thus, $yRx\subseteq J(R)$. \\
	Suppose that $R$ is NJ-reflexive. Let $X=\begin{pmatrix}
		x & m_1 \\
		0 & x
	\end{pmatrix}$, $Y=\begin{pmatrix}
		y & m_2 \\
		0 & y
	\end{pmatrix}\in \Gamma$ be such that $X\Gamma Y\subseteq N(\Gamma)$. Let $\alpha=\begin{pmatrix}
		r & m \\
		0 & r
	\end{pmatrix}$ be any element of $\Gamma$. As $R$ is NJ-reflexive, $yrx\in J(R)$, which implies that $yRx\subseteq J(R)$. Therefore, $Y\Gamma X\subseteq J(\Gamma )$.
\end{proof}

\begin{proposition}\label{Rn}
	The following are equivalent:
	\begin{enumerate}
		\item R is NJ-reflexive.
		\item $R_n=
		\left \lbrace
		\left(\begin{array}{lccccr}
			x & x_{12} & \dots  & x_{1(n-1)} & x_{1n}\\
			0 & x & \dots & x_{2(n-1)} & x_{2n}\\
			\vdots & \vdots &\ddots & \vdots & \vdots \\
			0 & 0 & \dots & x &x_{(n-1)n}\\
			0 & 0 & \dots & 0 & x \\
		\end{array}
		\right ):x, x_{ij}\in R, i<j\right \rbrace$ is NJ-reflexive.
	\end{enumerate}
\end{proposition} 

\begin{proof}
	$\mathbf{(1)\implies (2)}$. Consider the ideal $I=\left \lbrace
	\left(\begin{array}{lccccr}
		0 & x_{12} & \dots & x_{1n}\\
		0 & 0 & \dots & x_{2n}\\
		\vdots & \vdots &\ddots & \vdots\\
		0 & 0 & \dots & 0 \\
	\end{array}
	\right ): x_{ij}\in R\right \rbrace$ of $R_n$. It is easy to see that $I$ a nil ideal and $R_n/I\cong R$. Hence, by Theorem \ref{nilfrac}, $R_n$ is NJ-reflexive.\\
	
	$\mathbf{(2)\implies (1)}$. This implication follows from Proposition \ref{eRe}.  
\end{proof}
\begin{corollary}
	The following are equivalent:
	\begin{enumerate}
		\item R is NJ-reflexive.
		\item $R[x]/<x^n>$ is NJ-reflexive for any positive integer $n$.
	\end{enumerate}
\end{corollary}
\begin{proof}
	Observe that \\ $R[x]/<x^n>\cong  
	\left \lbrace
	\left(\begin{array}{lccccr}
		x_1 & x_{2} & x_3 & \dots & x_{n-1} & x_n\\
		0 & x_1 & x_2 & \dots &  x_{n-2} & x_{n-1}\\
		\vdots & \vdots  & \vdots & \vdots & \vdots \\
		0 & 0 & 0 & \dots & x_1 & x_2 \\
		0 & 0 & 0 & \dots & 0 & x_1
	\end{array}
	\right ):x_i\in R\right \rbrace$.
	Therefore, the proof can be derived from the proof of Proposition \ref{Rn}.
\end{proof}

Let $A$ be a ring (not necessarily with unity)  and $(R,R)$-bimodule   satisfying the conditions $(aw)r=a(wr)$, $(ar)w=a(rw)$ and $(ra)w=r(aw)$  for all $a,w\in A$ and $r\in R$. The \textit{ideal-extension}, also called \textit{Dorroh extension}, of $R$ by $A$ is defined to be the additive abelian group $I(R;A)=R\oplus A$ with multiplication given by $(r,a)(s,w)=(rs,rw+as+aw)$.

\quad Consider a ring homomorphism $\sigma: R\rightarrow R$. The ring of skew formal power series over $R$ is denoted by $R[[x,\sigma]]$, which is the ring of  all formal power series in $x$ with coefficients from $R$. The multiplication operation in $R[[x,\sigma]]$ is defined by the rule $xr=\sigma(r)x$ for all $r\in R$. It is a well-known fact that $J(R[[x,\sigma]])=J(R)+\langle x \rangle$ and $R[[x,\sigma]]$ is isomorphic to $I(R;\langle x \rangle)$. Observe that for each $a\in \langle x \rangle$, there exists $a_0\in \langle x \rangle$ satisfying $a+a_0+aa_0=0$.

\begin{proposition}\label{dor}
	Suppose that for any $a\in A$, there exists $a_0\in A$ such that $a+a_0+aa_0=0$. If R is NJ-reflexive, then the ideal extension $M=I(R;A)$ is NJ-reflexive.
	
\end{proposition}

\begin{proof}
	Assume that $R$ is NJ-reflexive. Consider $(r_1,s_1)M(r_2,s_2)\subseteq N(M)$ for some $(r_1,s_1), (r_2,s_2)\in M$. Thus, for any $(x,s)\in M$, we have $(r_1,s_1)(x,s)(r_2,s_2)\in N(M)$. Therefore, we can conclude that $(r_1xr_2,s_3)\in N(M)$ for some $s_3\in A$, implying that $r_1xr_2\in N(R)$. Since $R$ is NJ-reflexive, it follows that $r_2xr_1\in J(R)$. Moreover, note that $(r_2, s_2)(x, s)(r_1, s_1)=(r_2xr_1, s_4)$ for some $s_4\in A$. By hypothesis, $(0, A)\subseteq J(M)$. It can be easily shown that $(r_2xr_1, 0)\in J(M)$. Thus, $(r_2, s_2)(x, s)(r_1, s_1)=(r_2xr_1, s_4)\in J(M)$, which implies that $(r_2, s_2)M(r_1, s_1)\subseteq J(M)$. 
\end{proof}  

\quad As a consequence of Proposition \ref{dor}, the following corollary can be immediately deduced.

\begin{corollary}
	Let $\sigma:R\rightarrow R$ be a ring homomorphism. If  R is NJ-reflexive, then $R[[x,\sigma]]$ is NJ-reflexive.
\end{corollary}

It may be of interest to determine whether the polynomial ring over an NJ-reflexive ring is also NJ-reflexive. An example in this regard is provided below.

\begin{example}

	Let $R=K+yM_2(K)[[y]]$, where   $K$ is a division ring. Note that $J(R)=yM_2(K)[[y]]$, so $R$ is a local ring. By Theorem \ref{qduo}, $R$ is NJ-reflexive. We show that $R[x]$ is not NJ-reflexive.
	Suppose, to the contrary, that $R[x]$ is NJ-reflexive. Observe that $(E_{12}y)R[x](E_{12}y)\subseteq N(R[x]))$. Since $R[x]$ is NJ-reflexive,  $(E_{12}y)R[x](E_{12}y)\subseteq J(R[x])$. This implies that $E_{12}y\in J(R[x])$. If there exists a non-negative integer $i$ such that $E_{12}y^i\in N^*(R)$, then $M_2(K)y^{i+2}=(M_2(K)E_{12}M_2(K))y^{i+2}=(M_2(K)y)E_{12}y^i(M_2(K)y)\subseteq N^*(R)$.  But this is a contradiction to $y^{i+1}\in M_2(K)y^{i+2}$ being a non-nilpotent element. Thus, $E_{12}y^i\notin N^*(R)$ for all non-negative integers $i$. According to \cite[Theorem 2]{radpol}, $J(R[x])=L[x]$ for some nil ideal $L$. So $E_{12}y\notin J(R[x])$, a contradiction. Therefore, $R[x]$ is not NJ-reflexive.
\end{example}

\end{document}